\documentclass[12pt,twoside,openany]{paper}
\usepackage{enumerate,amsmath,graphics,amssymb,graphicx,amscd,xypic,amscd,amsbsy,multirow,float,booktabs,verbatim}
\newcommand{\QED}{\hspace*{\fill}\rule{2.5mm}{2.5mm}}
\newtheorem{theorem}{Theorem}[section]

\newenvironment{proof}{\noindent{\bf Proof\ }}{\QED\\}
\newcommand{\R}{\mathbb{R}}

\newtheorem{lemma}{Lemma}[section]

\begin{document}
\begin{center}
\vspace{0.5cm} {\large \bf ``Quantiles Equivariance"}\\
\vspace{1cm} Reza Hosseini, University of British Columbia\\
333-6356, Agricultural Road, Vancouver,\\
 BC, Canada, V6T1Z2\\
reza1317@gmail.com
\end{center}

\section{Abstract}
It is widely claimed that the quantile function is equivariant under
increasing transformations. We show by a counterexample that this is
not true (even for strictly increasing transformations). However, we
show that the quantile function is equivariant under left continuous
increasing transformations. We also provide an equivariance relation
for continuous decreasing transformations. In the case that the
transformation is not continuous, we show that while the transformed
quantile at  p can be arbitrarily far from the quantile of the
transformed at p (in terms of absolute difference), the probability
mass between the two is zero. We also show by an example that
weighted definition of the median is not equivariant under even
strictly increasing
continuous transformations.\\

\noindent Keywords: Quantile, quantile function, distribution
function, equivariance, continuous transformation, increasing
transformation

\section{Introduction}

The traditional definition of quantiles for a random variable $X$
with distribution function $F$,
\[lq_X(p)=\inf \{x|F(x) \geq p\},\]
appears in classic works as \cite{parzen-1979}. We call this the
``left quantile function''. In some books (e.g. \cite{rychlik})
the quantile is defined as
\[rq_X(p)=\inf \{x|F(x) > p\}=\sup \{x|\; F(x) \leq p\},\]
this is what we call the ``right quantile function''. Also in
robustness literature people talk about the upper and lower
medians which are a very specific case of these definitions.
Hosseini in \cite{reza-phd} considers both definitions, explore
their relation and show that considering both has several
advantages.

It is widely claimed that (e.g. Koenker in
\cite{quantile-regression-koenker} or Hao and Naiman in
\cite{quantile-regression-hao}) the traditional quantile function
is equivariant under monotonic transformations. We show that this
does not hold even for strictly increasing functions. However, we
prove that the traditional quantile function is equivariant under
non-decreasing left continuous transformations. We also show that
the right quantile function is equivariant under non-decreasing
right continuous transformations. A similar neat result is found
for continuous decreasing transformations using the Quantile
Symmetry Theorem also proved by Hosseini in \cite{reza-phd}. We
state this theorem later when we need it. Hosseini in
\cite{reza-phd}, proved the two following useful lemmas using the
definition of quantiles. We will use some of the items in these
lemmas in our proofs.

\begin{lemma} (Quantile Properties Lemma) Suppose $X$ is a random
variable on the probability space $(\Omega,\Sigma,P)$ with
distribution function $F$:

\begin{enumerate}[a)]
\item $F(lq_F(p))\geq p$. \item $lq_F(p) \leq rq_F(p)$.
\item $p_1<p_2 \Rightarrow rq_F(p_1)\leq lq_F(p_2)$.
\item $rq_F(p)=\sup\{x|F(x)\leq p\}$.
\item $P(lq_F(p)<X<rq_F(p))=0$. i.e. $F$ is flat in the interval $(lq_F(p),rq_F(p))$. \item $P(X<
rq_F(p)) \leq p$.
\item If $lq_F(p)<rq_F(p)$ then $F(lq_F(p))=p$ and hence $P(X\geq rq_F(p))=1-p$. \item $lq_F(1)>-\infty,rq_F(0)<\infty$
and $P(rq_F(0) \leq X \leq lq_F(1))=1$.
\item $lq_F(p)$ and $rq_F(p)$ are non-decreasing functions of $p$.
\item If $P(X=x)>0$ then $lq_F(F(x))=x.$ \item $x<lq_F(p) \Rightarrow F(x)<p$ and
$x>rq_F(p) \Rightarrow F(x)>p.$
\end{enumerate}

\label{quantile-properties}
\end{lemma}

\begin{lemma}(Quantile Value Criterion Lemma)
\begin{enumerate}[a)]

\item $lq_F(p)$ is the only $a$ satisfying (i) and (ii), where\\
(i) $F(a) \geq p$,\\
(ii) $x<a \Rightarrow F(x)<p.$

\item $rq_F(p)$ is the only $a$ satisfying (i) and (ii), where\\
(i) $x<a \Rightarrow F(x) \leq p$,\\
 (ii) $x>a \Rightarrow F(x)>p$.
\end{enumerate}
\label{lemma-quantile-value-charac}
\end{lemma}

\begin{proof}
\begin{enumerate}[a)]

\item Both properties hold for $lq_F(p)$ by the previous lemma. If
both $a<b$ satisfy them, then $F(a) \geq p$ by (i). But since $b$
satisfies the properties and $a<b$, by (ii), $F(a)<p$ which is a
contradiction.

\item Both properties hold for $rq_F(p)$ by the previous lemma. If
both $a<b$ satisfy them, then we can get a contradiction similar
to above.

\end{enumerate}
\end{proof}

It is customary to use weighted procedures to define the quantiles
of a data vector. The most widely used example is the definition
of median when for the sorted data vector $x=(x_1,\cdots,x_n)$,
$n$ is even, in which case the median is defined to be
$\frac{x_{\frac{n}{2}}+x_{\frac{(n+2)}{2}}}{2}$. We start by an
example that shows with this definition the median is not
equivariant even under continuous strictly increasing
transformation (a continuous re-scaling of data).

\begin{example}
A supervisor asked 2 graduate students to summarize the following
data regarding the intensity of the earthquakes in a specific
region:

{\tiny
\begin{table}[H]
  \centering  \footnotesize
  \begin{tabular}{lcccccccc}
\toprule[1pt]
row number & $M_L$ (Richter) & $A$ (shaking amplitude)\\
\midrule[1pt]
 1& 4.21094&  $1.62532 \times 10^4$\\
 2& 4.69852&  $4.99482 \times 10^4$\\
 3& 4.92185&  $8.35314 \times 10^4$\\
 4& 5.12098& $13.21235 \times 10^4$\\
 5& 5.21478& $16.39759 \times 10^4$\\
 6& 5.28943& $19.47287 \times 10^4$\\
 7& 5.32558& $21.16313 \times 10^4$\\
 8& 5.47828& $30.08015 \times 10^4$\\
 9& 5.59103& $38.99689 \times 10^4$\\
10& 5.72736& $53.37772 \times 10^4$\\
\bottomrule[1pt]
\end{tabular}
\caption{Earthquakes intensities}
 \label{table:quantile-def-richter}
\end{table}
} Earthquake intensity is usually measured in $M_L$ scale, which
is related to $A$ by the following formula:

\[M_L=\log_{10}A.\]
In the data file handed to the students (Table
\ref{table:quantile-def-richter}), the data is sorted with respect
to $M_L$ in increasing order from top to bottom. Hence the data is
arranged decreasingly with respect to $A$ from top to bottom.

The supervisor asked two graduate students to compute the center
of the intensity of the earthquakes using this dataset. One of the
students used $A$ and the usual definition of median and so
obtained

\[( 16.39759 \times 10^4+19.47287 \times 10^4)/2=17.93523 \times 10^{4}.\]

The second student used the $M_L$ and the usual definition of
median to find

\[(5.21478+5.28943)/2=5.252105.\]

When the supervisor saw the results he figured that the students
must have used different scales. Hence he tried to make the scales
the same by transforming the second student's result

\[10^{5.252105}= 17.86920 \times 10^{4}.\]
To his surprise the results were not quite the same. He was
bothered to notice that the definition of median is not
equivariant under the change of scale which is continuous strictly
increasing.
\end{example}

\section{Equivariance property of quantile functions}

\begin{example}(Counter example for Koenker--Hao claim)
Suppose $X$ is distributed uniformly on [0,1]. Then
$lq_X(1/2)=1/2.$ Now consider the following strictly increasing
transformation
\[\phi(x)=\begin{cases}x & -\infty<x<1/2\\
x+5 & x \geq 1/2 \end{cases}.\] Let $T=\phi(X)$ then the
distribution of $T$ is given by

\[P(T \leq t)=
\begin{cases}0 &  t\leq 0\\
t & 0< t \leq 1/2\\
1/2 & 1/2<t \leq 5+1/2\\
t-5 & 5+1/2<t \leq 5+1\\
1 & t>5+1\end{cases}.\]

It is clear form above that $lq_T(1/2)=1/2 \neq
\phi(lq_X(1/2))=\phi(1/2)=5+1/2.$

\end{example}

We start by defining

\[\phi^{\leq}(y)=\{x|\phi(x) \leq y\},\;\phi^{\star}(y)=\sup \phi^{\leq}(y),\]
and
\[\phi^{\geq}(y)=\{x|\phi(x) \geq y\},\;\phi_{\star}(y)=\inf \phi^{\geq}(y).\]
Then we have the following lemma.

\begin{lemma}
Suppose $\phi$ is non-decreasing.

\begin{enumerate}[a)]

\item If $\phi$ is left continuous then
\[\phi(\phi^{\star}(y)) \leq y.\]

\item If $\phi$ is right continuous then
\[\phi(\phi_{\star}(y)) \geq y.\]

\end{enumerate}

\begin{proof}

\begin{enumerate}[a)]

\item Suppose $x_n \uparrow \phi^{\star}(y)$ a strictly increasing
sequence. Then since $x_n<\phi^{\star}(y)$, we conclude  $x_n \in
\phi^{\leq}(y) \Rightarrow \phi(x_n) \leq y.$ Hence $\lim_{n
\rightarrow \infty} \phi(x_n) \leq y.$ But by left continuity
$\lim_{n \rightarrow \infty}\phi(x_n)=\phi(\phi^{\star}(y))$.

\item Suppose $x_n \downarrow \phi_{\star}(y)$ a strictly
decreasing sequence. Then since $x_n>\phi_{\star}(y),$ we conclude
$x_n \in \phi^{\geq}(y) \Rightarrow \phi(x_n) \geq y.$ Hence
$\lim_{n \rightarrow \infty} \phi(x_n) \geq y.$ But by right
continuity $\lim_{n \rightarrow
\infty}\phi(x_n)=\phi(\phi_{\star}(y))$.

\end{enumerate}
\end{proof}

\end{lemma}

\begin{theorem} (Quantile Equivariance Theorem)
Suppose $\phi:\R \rightarrow \R$ is non-decreasing.
\begin{enumerate}[a)]
\item If $\phi$ is left continuous then
\[lq_{\phi(X)}(p)=\phi(lq_X(p)).\]

\item If $\phi$ is right continuous then
\[rq_{\phi(X)}(p)=\phi(rq_X(p)).\]

\end{enumerate}
\label{theo-quantile-equiv}
\end{theorem}

\begin{proof}

\begin{enumerate}[a)]

\item We use Lemma \ref{lemma-quantile-value-charac} to prove
this. We need to show (i) and (ii) in that lemma for
$\phi(lq_X(p))$. First note that (i) holds since
\[F_{\phi(X)}(\phi(lq_X(p)))=P(\phi(X)\leq \phi(lq_X(p))) \geq P(X \leq lq_X(p)) \geq p.\]
For (ii) let $y<\phi(lq_X(p))$. Then we want to show that
$F_{\phi(X)}(y)<p$.
 It is sufficient to show $\phi^{\star}(y)<lq_X(p).$ Because then

 \[P(\phi(X) \leq y) \leq P(X \leq \phi^{\star}(y))<p.\]
 To prove $\phi^{\star}(y)<lq_X(p)$, note that by the previous lemma
 \[\phi(\phi^{\star}(y)) \leq y <\phi(lq_X(p)).\]

\item We use Lemma \ref{lemma-quantile-value-charac} to prove
this. We need to show (i) and (ii) in that lemma for
$\phi(rq_X(p))$. To show (i) note that if $y<\phi(rq_X(p))$,
\[P(\phi(X) \leq y) \leq P(\phi(X) < \phi(rq_X(p))) \leq P(X< rq_X(p)) \leq p.\]
To show (ii), suppose $y>\phi(rq_X(p))$. We only need to show
$\phi_{\star}(y)>rq_X(p)$ because then
\[P(\phi(X) \leq y) \geq P(X < \phi_{\star}(y)) >p.\]
But by previous lemma $\phi(\phi_{\star}(y)) \geq y >
\phi(rq_X(p))$. Hence $\phi_{\star}(y)>rq_X(p)$.
\end{enumerate}
\end{proof}

In order to find an equivariance under decreasing transformations
we need the Quantile Symmetry Theorem proved by Hosseini in
\cite{reza-phd}.

\begin{theorem}(Quantile Symmetry Theorem) Suppose $X$ is a random variable and $p \in [0,1]$.
Then

\[lq_X(p)=-rq_{-X}(1-p).\]

\label{theo-quantile-symmetry}
\end{theorem}

\begin{theorem} (Decreasing transformation equivariance)\\
a) Suppose $\phi$ is non-increasing and  right continuous on $\R$.
Then
\[lq_{\phi(X)}(p)=\phi(rq_X(1-p)).\]
b) Suppose $\phi$ is non-increasing and  left continuous on $\R$.
Then
\[rq_{\phi(X)}(p)=\phi(lq_X(1-p)).\]
 \label{theo-decreas-equiv}
\end{theorem}

\begin{proof}
a) By the Quantile Symmetry Theorem, we have
\[lq_{\phi(X)}(p)=-rq_{-\phi(X)}(1-p).\]
But $-\phi$ is  non-decreasing right continuous, hence the above
is equal to
\[-(-\phi(rq_X(1-p)))=\phi(rq_X(1-p)).\]
b) By the Quantile symmetry Theorem
\[rq_{\phi(X)}(p)=-lq_{-\phi(X)(1-p)}=-(-\phi(lq_X(1-p)))=\phi(lq_X(p)),\]
since $-\phi$ is non-decreasing and left continuous.
\end{proof}

\section{The non-continuous case}

We showed by an example that the equivariance property does not
hold for increasing transformations that are not continuous.
However we show here that the transformed quantile is not that
much off at the end in a specific sense. We start by a lemma.
\begin{lemma}
Let $X$ be a random variable. Then
\[[lq_X(p),rq_X(p)]=\{y|\;F_X^o(y)\leq p,\; F_X(y) \geq p\},\]
where $F^o_X(x)=P(X<x),\; F_X(x)=P(X \leq x).$
\end{lemma}

\begin{proof}
By the Quantile Property Lemma (a)
\[F(a)\geq F(lq_X(p)) \geq p,\]
for all $a \geq lq_X(p)$. Now note that by  the Quantile Value
Criterion Lemma, Part (b), we have
\[F^o_X(a) \leq F^o_X(p)=\lim_{x \rightarrow rq_X(p)^+} F(x) \leq p,\]
for all $a \leq rq_X(p)$, which shows
\[[lq_X(p),rq_X(p)] \subset \{y|\;F_X^o(y)\leq p,\; F_X(y) \geq p\}.\]
To prove the converse, suppose $y<lq_X(p)$ then $F(y)<p$ by
Quantile Value Criterion Lemma, Part (a) and hence $y \notin
\{y|\;F_X^o(y)\leq p, F_X(y) \geq p\}.$ Similarly for $y>rq_X(p)$
take $y>z>rq_X(p)$ by the Part (b) of the lemma
\[F_X^o(y)\leq F_X(y)>p.\]
Hence  $y \notin \{y|\;F_X^o(y)\leq p, F_X(y) \geq p\}.$
\end{proof}

\begin{lemma} (Equivariance under non-decreasing transformations)
Suppose $X$ is a random variable with distribution function  $F$
and $\phi:\R \rightarrow \R$ a non-decreasing transformation on
$\R$. Also let $Y=\phi(X)$. Then\\
a) $\phi(lq_X(p)) \in [lq_Y(p),rq_Y(p)]$\\
b) $\phi(rq_X(p)) \in [lq_Y(p),rq_Y(p)].$

\end{lemma}

\begin{proof}
Note that
\[F_Y^o(\phi(lq_X(p)))=P(\phi(X)<\phi(lq_X(p)))\leq P(X < lq_X(p)) \leq p,\]
and
\[F_Y(\phi(lq_X(p)))=P(\phi(X) \leq \phi(lq_X(p)))\geq P(X \leq lq_X(p)) \geq p.\]
Hence proving a) by the previous lemma and b) is similar.
\end{proof}

\noindent {\bf Remark.} If we consider the ``probability loss
function'' defined as
\[\delta_Y(a,b)=P(a<Y<b)+P(b<Y<a),\]
then the above lemma states that
\[\delta_Y(\phi(lq_X(p)),lq_Y(p))=0,\] and
\[\delta_Y(\phi(rq_X(p)),rq_Y(p))=0.\] Hosseini in \cite{reza-phd}
studied this loss function and used it in approximating quantiles
in large datasets.

\bibliographystyle{plain}
\bibliography{mybibreza}

\end{document}